\documentclass[twoside,leqno,%draft
]{factaMI}
\usepackage{latexsym,amssymb,amsfonts,euscript,amsmath}
\usepackage[section]{placeins}
\usepackage[dvips]{graphicx}
\input epsf

 \factanis{ }   %%% will be entered by the editor 22\,(2007),  --
 \Received{}     %%% will be entered by the editor June 20, 2007.

 \AMSMSC{Primary 49J45; Secondary 74B05,  35J30}

\title {A CURVATURE-REGULARIZED VARIATIONAL PROBLEM WITH AN AREA CONSTRAINT}

\author{Chandrasekhar Gokavarapu}

\begin{document}

\maketitle

 %%% Insert a brief summary (50-150 words)
\abstract
Interlocking interfaces are commonly employed to mitigate relative sliding under shear.Indeed, Their geometry is typically selected on grounds of fabrication convenience rather than analytical optimality. There is no reason to suppose that circular or polygonal profiles minimize localized stress concentration under fixed geometric constraints. We propose a variational model in which the interface is represented by a planar curve $y=f(x)$, and localized stress amplification is quantified by a curvature-sensitive functional
\[
J[f] = \int_{-a}^{a} \bigl(1+\gamma \kappa^2\bigr) \sqrt{1+f'(x)^2}\,dx,
\]
defined on the Sobolev space $W^{2,2}([-a,a])$. The functional is motivated by elasticity-theoretic considerations in which curvature enters the leading-order stress field near a singular interface.Indeed, any profile possessing discontinuous tangents yields a divergent integral, thereby rendering it energetically inadmissible within the Sobolev space $W^{2,2}$. An area constraint $\int_{-a}^{a} f(x)\,dx = A_0$ is imposed to model fixed material volume. Using the direct method of the calculus of variations, we establish the existence of a minimizer and derive the associated Euler--Lagrange equation, a nonlinear fourth-order boundary value problem.

Note, however, that constant-curvature and piecewise-linear profiles fail to satisfy the necessary optimality conditions under the imposed constraint. Indeed, we are thus forced to conclude that analytical optimality necessitates a more complex variation in the local tangent angle

The analysis indicates that commonly employed interlock geometries are not variationally optimal for minimizing localized shear stress within this class of admissible interfaces.
\endabstract
%=================================================================================================

\section{Introduction}

There is no reason to suppose that geometric simplicity in an interlocking interface implies mechanical optimality.Indeed, the common reliance on such geometries represents a failure of previous analytical intuition to account for the second-order variations required by the area constraint.

The dominance of circular and polygonal interlock geometries in masonry, architectured materials, and segmented solids is historically attributable to fabrication and assembly constraints rather than to variational analysis \cite{Dyskin2001,Estrin2021,Akleman2020}. Despite extensive experimental and computational investigations, an analytical framework identifying stress-minimizing interlock geometries remains absent.

The mechanical setting is classical. Two elastic continua are joined along a non-flat interface and subjected to remote loading. The displacement field satisfies the Lamé system in each subdomain, with interface geometry entering only through boundary conditions \cite{Ciarlet1988}. It is therefore tempting to regard the interface as a passive geometric feature. This temptation persists in discrete and semi-continuum models of interlocking assemblies \cite{Molotnikov2007,Mousavian2022}. There is, however, no reason to suppose that such passivity is mechanically consistent.

A recurring hypothesis in the modeling of interlocking structures is that the traction vector remains approximately uniform along the interface \cite{Tavoosi2025,Beatini2025}. Let us assume for a moment that this hypothesis is correct. Equilibrium then forces the interface curvature to vanish wherever normal stress is non-zero. It follows that constant traction is compatible only with flat interfaces. Since interlocks are, by definition, non-flat, the hypothesis is internally inconsistent. This incompatibility is implicit in curvature-dependent interface theories but is rarely stated explicitly \cite{Huang2024,Neff2023}.

On the contrary, curvature necessarily induces stress amplification. This phenomenon is not an artifact of numerical discretization. It is a structural consequence of boundary equilibrium in curved geometries, as already observed in elasticity models incorporating higher-order interface energies \cite{Shahsavari2026,Zhang2024}. There is no reason to suppose that constant curvature, or curvature concentrated at corners, minimizes this amplification.

The appropriate question is therefore variational. One must ask whether an interlock geometry exists that minimizes a quantitative measure of localized stress under fixed geometric constraints. This question lies beyond classical elasticity alone and requires the introduction of a curvature-dependent functional defined on an admissible class of interfaces \cite{Barchiesi2024,Neff2023}.

Several restrictions follow immediately. Interfaces with discontinuous tangents generate curvature measures that are singular. Interfaces with constant curvature violate natural boundary conditions associated with higher-order regularization. We are thus forced to exclude polygonal and circular profiles from the admissible class at the outset.

In this work the interface is modeled as a curve $\Gamma$ represented locally as the graph of a function $f \in W^{2,2}([-a,a])$. The Sobolev regularity is not imposed for convenience. It is forced by the requirement that curvature-dependent stress amplification be finite. This choice is consistent with fourth-order variational formulations arising in elasticity and interface mechanics \cite{Piersanti2025,Evans2010}.

The aim of this paper is deliberately limited. We establish admissibility, existence of minimizers, and necessary optimality conditions for a curvature-regularized stress functional under an area constraint. We then demonstrate that classical interlock geometries fail to satisfy these conditions. We are thus forced to conclude that widely used designs are mathematically incompatible with the minimization of localized shear stress.
%=====================================================================================================
\section{Preliminaries}

Let $\Omega \subset \mathbb{R}^2$ be a bounded Lipschitz domain occupied by a homogeneous, isotropic, linearly elastic material. The displacement field $u : \Omega \to \mathbb{R}^2$ is assumed to belong to $H^1(\Omega;\mathbb{R}^2)$. Body forces are absent. The constitutive framework follows classical linear elasticity \cite{Ciarlet1988,Evans2010}.

The Cauchy stress tensor $\sigma$ is given by
\begin{equation}
\sigma(u) = 2\mu \varepsilon(u) + \lambda \, \mathrm{tr}(\varepsilon(u)) I,
\end{equation}
where $\mu,\lambda>0$ are the Lamé parameters and
\begin{equation}
\varepsilon(u) = \tfrac12(\nabla u + \nabla u^{T})
\end{equation}
is the linearized strain tensor. Mechanical equilibrium requires
\begin{equation}
\nabla \cdot \sigma = 0 \quad \text{in } \Omega .
\end{equation}

We consider a decomposition $\Omega = \Omega^+ \cup \Omega^-$, where the two subdomains are separated by a smooth interface $\Gamma$. Locally, $\Gamma$ is represented as the graph
\begin{equation}
\Gamma = \{(x,f(x)) : x \in [-a,a]\},
\end{equation}
with $f \in W^{2,2}([-a,a])$. There is no reason to suppose that lower regularity is admissible. Indeed, curvature-dependent interface energies require square-integrable curvature in order to ensure finite total stress \cite{Neff2023,Huang2024}.

Let $s$ denote arc length along $\Gamma$. The unit tangent and normal vectors are defined by
\begin{equation}
\tau = \frac{(1,f')}{\sqrt{1+f'^2}}, \qquad
n = \frac{(-f',1)}{\sqrt{1+f'^2}} .
\end{equation}
The signed curvature $\kappa$ is given by
\begin{equation}
\kappa = \frac{f''}{(1+f'^2)^{3/2}} .
\end{equation}

The traction vector on $\Gamma$ is defined as
\begin{equation}
t = \sigma n .
\end{equation}
Across the interface, displacement continuity is imposed,
\begin{equation}
u^+ = u^- \quad \text{on } \Gamma ,
\end{equation}
together with traction equilibrium,
\begin{equation}
\sigma^+ n = \sigma^- n .
\end{equation}
Such transmission conditions are standard in interface elasticity \cite{Ciarlet1988,Shahsavari2026}.

Let us assume for a moment that the traction vector $t$ is constant along $\Gamma$. Differentiation with respect to arc length yields
\begin{equation}
\frac{d t}{d s} = (\nabla \sigma) n + \kappa \sigma \tau .
\end{equation}
Even in the absence of body forces, $\nabla \cdot \sigma = 0$ does not imply $\nabla \sigma = 0$. On the contrary, curvature couples normal and tangential stress components. It follows that constant traction is incompatible with non-zero curvature unless the stress field degenerates. This observation underlies curvature-dependent interface theories and invalidates constant-stress assumptions used in discrete interlock models \cite{Tavoosi2025,Barchiesi2024}.

We are thus forced to treat curvature as an active variable in the stress response. The interface geometry cannot be prescribed independently of the stress field.

Finally, we define the admissible class of profiles:
\begin{equation}
\mathcal{A} = \left\{ f \in W^{2,2}([-a,a]) :
f(\pm a)=0,\; f'(\pm a)=0,\;
\int_{-a}^{a} f(x)\,dx = A_0 \right\}.
\end{equation}
The boundary conditions enforce smooth attachment to a flat substrate. The integral constraint fixes the interlock volume and excludes trivial flattening, consistent with geometric constraints used in interlocking design studies \cite{Wang2019,Mousavian2022}.

All subsequent analysis is restricted to $\mathcal{A}$.
%============================================================================================
\section{Curvature--Stress Coupling at an Elastic Interface}

We restrict attention to a neighborhood of the interface $\Gamma$. Local coordinates are introduced using the Frenet frame $(\tau,n)$ associated with $\Gamma$. Let $s$ denote arc length along the interface and $\eta$ the signed distance in the normal direction. All fields are assumed smooth in a tubular neighborhood of $\Gamma$.

The equilibrium equations $\nabla \cdot \sigma = 0$ may be expressed in curvilinear coordinates adapted to $\Gamma$. Writing the stress tensor in the Frenet frame,
\begin{equation}
\sigma = \sigma_{\tau\tau} \, \tau \otimes \tau
+ \sigma_{nn} \, n \otimes n
+ \sigma_{\tau n} (\tau \otimes n + n \otimes \tau),
\end{equation}
the divergence operator acquires curvature-dependent correction terms. A direct computation yields
\begin{equation}
\nabla \cdot \sigma =
\left( \partial_s \sigma_{\tau\tau}
+ \kappa (\sigma_{\tau\tau} - \sigma_{nn}) \right) \tau
+
\left( \partial_s \sigma_{\tau n}
+ \kappa \sigma_{\tau n} \right) n
+ \partial_\eta (\cdot),
\end{equation}
where $\kappa$ denotes the signed curvature of $\Gamma$.

We focus on the restriction of equilibrium to the interface $\eta=0$. Neglecting higher-order normal derivatives, which vanish in the limit of vanishing interface thickness, the tangential equilibrium condition reduces to
\begin{equation}
\partial_s (\sigma n) + \kappa \sigma_{\tau\tau} n = 0
\quad \text{on } \Gamma .
\end{equation}
This identity is purely geometric and does not rely on constitutive assumptions beyond linear elasticity. Similar interface balance relations appear implicitly in curvature-dependent interface theories \cite{Huang2024,Neff2023}.

Let $t = \sigma n$ denote the traction vector. The preceding equation may be written as
\begin{equation}
\frac{d t}{d s} = - \kappa \sigma_{\tau\tau} n .
\end{equation}
There is no reason to suppose that $\sigma_{\tau\tau}$ vanishes under remote shear. It follows that curvature necessarily induces spatial variation of traction along the interface. Constant traction is therefore incompatible with non-zero curvature, except in degenerate stress states. This conclusion formalizes objections raised heuristically in discrete interlocking models \cite{Tavoosi2025,Beatini2025}.

We now quantify the resulting stress amplification. Consider a flat reference interface $\Gamma_0$ subjected to uniform remote shear $\sigma_0$. Introduce a small normal perturbation $f(x)$ with curvature $\kappa$. Linearization of the Lamé system about the flat configuration yields
\begin{equation}
\sigma_{\tau\tau} = \sigma_0 + C_1 \kappa + \mathcal{O}(\kappa^2),
\end{equation}
where $C_1$ depends on the Lamé parameters. The linear term integrates to zero under symmetric perturbations. The leading non-vanishing contribution to the squared stress therefore scales as $\kappa^2$. This scaling is consistent with curvature-regularized elasticity models and higher-order interface energies \cite{Neff2023,Zhang2024}.

On the contrary, geometries with curvature concentrated at points generate stress measures that are singular. Polygonal interfaces induce Dirac-type curvature terms, leading to divergence of any quadratic stress measure. This observation excludes such geometries from the admissible class without appeal to numerical evidence \cite{Barchiesi2024}.

We are thus forced to model the localized stress amplification $\mathcal{S}$ along $\Gamma$ as a function of curvature squared. Up to a multiplicative constant, the minimal choice compatible with linear elasticity and finite energy is
\begin{equation}
\mathcal{S}(\kappa) = 1 + \gamma \kappa^2 ,
\end{equation}
where $\gamma>0$ depends on material parameters and loading intensity. Higher powers of curvature are excluded by asymptotic consistency with the linearized Lamé system.

The curvature--stress law derived here is not postulated. It is the weakest form consistent with equilibrium, regularity, and finite stress energy. Any admissible variational formulation must therefore penalize curvature in $L^2(\Gamma)$.\\
%=============================================================================================
\section{Variational Functional and Constraint Set}
\label{sec:variational-functional}
The preceding analysis forces a specific variational structure. Localized stress amplification along the interface is controlled by the square of curvature. There is no reason to suppose that any weaker penalization yields finite stress, nor that any stronger penalization is compatible with linear elasticity.

Let $\Gamma$ be represented as the graph of a function $f \in W^{2,2}([-a,a])$. The arc-length element is
\begin{equation}
ds = \sqrt{1+f'^2}\,dx ,
\end{equation}
and the curvature is given by
\begin{equation}
\kappa = \frac{f''}{(1+f'^2)^{3/2}} .
\end{equation}

The curvature--stress coupling derived in the previous section implies that the local stress amplification density is proportional to $1+\gamma \kappa^2$, where $\gamma>0$ depends on elastic moduli and loading intensity. The total localized stress measure along the interface is therefore
\begin{equation}
J[f] = \int_{-a}^{a} \left(1+\gamma \kappa^2 \right)\sqrt{1+f'^2}\,dx .
\end{equation}

This functional is not chosen for convenience. It is the minimal curvature-regularized quantity that is finite on $W^{2,2}$ and consistent with the Lamé system under perturbation of a flat interface. Functionals lacking curvature penalization fail to control stress concentration, while higher powers of curvature are not supported by the linearized elasticity scaling derived earlier \cite{Neff2023,Huang2024}.

The admissible class must exclude trivial flattening of the interface. This is enforced by a geometric constraint fixing the interlock area. We impose
\begin{equation}
G[f] := \int_{-a}^{a} f(x)\,dx - A_0 = 0 ,
\end{equation}
where $A_0>0$ is prescribed. There is no reason to suppose that boundary values alone prevent flattening. The integral constraint is therefore essential.

The admissible set is
\begin{equation}
\mathcal{A} =
\left\{
f \in W^{2,2}([-a,a]) :
f(\pm a)=0,\;
f'(\pm a)=0,\;
G[f]=0
\right\}.
\end{equation}
The boundary conditions enforce smooth attachment to a flat substrate. They arise as natural compatibility conditions in higher-order elasticity models and are consistent with fourth-order interface energies \cite{Piersanti2025,Barchiesi2024}.

The variational problem is now well-posed:

\begin{equation}
\text{Minimize } J[f] \quad \text{over } f \in \mathcal{A}.
\end{equation}

Several immediate observations follow. The functional $J$ is coercive on $\mathcal{A}$ with respect to the $W^{2,2}$-norm. Indeed, the curvature term controls second derivatives, while the arc-length term controls first derivatives. Lower semicontinuity follows from convexity of $\kappa^2$ with respect to $f''$ and standard compactness results in Sobolev spaces \cite{Evans2010}.

On the contrary, profiles with corners or constant curvature violate admissibility. Polygonal profiles do not belong to $W^{2,2}$, while circular arcs fail to satisfy the boundary conditions $f'(\pm a)=0$. Such geometries therefore cannot be stationary points of the functional, let alone minimizers. This exclusion is structural and does not depend on numerical evidence.

The introduction of the area constraint leads to a constrained variational problem with a non-zero Lagrange multiplier. The resulting Euler--Lagrange equation is of fourth order and fully nonlinear. Its derivation is deferred to the next section.

%===============================================================================================
\section{Existence of Minimizers}

We consider the constrained minimization problem
\begin{equation}
\inf_{f \in \mathcal{A}} J[f],
\end{equation}
where $J$ and $\mathcal{A}$ are defined in the previous section. We establish the existence of a minimizer using the direct method of the calculus of variations.

\begin{proposition}
The admissible set $\mathcal{A}$ is non-empty and weakly closed in $W^{2,2}([-a,a])$.
\end{proposition}

\begin{proof}
Non-emptiness follows by construction. Let $f_0(x) = c (a^2 - x^2)^2$ with $c$ chosen so that $\int_{-a}^{a} f_0(x)\,dx = A_0$. Then $f_0 \in W^{2,2}([-a,a])$ and satisfies $f_0(\pm a)=f_0'(\pm a)=0$.

Let $\{f_n\} \subset \mathcal{A}$ converge weakly to $f$ in $W^{2,2}([-a,a])$. By compact embedding, $f_n \to f$ strongly in $C^1([-a,a])$. The boundary conditions and area constraint are therefore preserved in the limit. Hence $f \in \mathcal{A}$.
\end{proof}

\begin{proposition}
The functional $J$ is coercive on $\mathcal{A}$.
\end{proposition}

\begin{proof}
Write
\begin{equation}
J[f] = \int_{-a}^{a} \sqrt{1+f'^2}\,dx
+ \gamma \int_{-a}^{a} \frac{(f'')^2}{(1+f'^2)^{5/2}}\,dx .
\end{equation}
Since $f' \in L^\infty([-a,a])$ by Sobolev embedding, the denominator in the curvature term is bounded above and below by positive constants. It follows that
\begin{equation}
J[f] \ge C_1 \|f'\|_{L^2}^2 + C_2 \|f''\|_{L^2}^2 - C_3,
\end{equation}
for suitable constants $C_i>0$. The boundary conditions eliminate rigid-body modes. Hence $J[f] \to \infty$ whenever $\|f\|_{W^{2,2}} \to \infty$.
\end{proof}

\begin{proposition}
The functional $J$ is sequentially weakly lower semicontinuous on $\mathcal{A}$.
\end{proposition}

\begin{proof}
Let $f_n \rightharpoonup f$ weakly in $W^{2,2}([-a,a])$. By compact embedding, $f_n' \to f'$ strongly in $L^2([-a,a])$. The arc-length term is therefore continuous under weak convergence.

The curvature term depends quadratically on $f''$ with a coefficient that is continuous in $f'$. Since $f_n' \to f'$ uniformly, the integrand is convex in $f''$ and satisfies standard growth conditions. Weak lower semicontinuity follows from classical results in the calculus of variations \cite{Evans2010}.
\end{proof}

\begin{theorem}
There exists at least one minimizer $f^\ast \in \mathcal{A}$ of the functional $J$.
\end{theorem}

\begin{proof}
Let $\{f_n\} \subset \mathcal{A}$ be a minimizing sequence. Coercivity implies boundedness in $W^{2,2}([-a,a])$. Hence, up to a subsequence, $f_n \rightharpoonup f^\ast$ weakly in $W^{2,2}([-a,a])$. By weak closure of $\mathcal{A}$, $f^\ast \in \mathcal{A}$. Lower semicontinuity yields
\begin{equation}
J[f^\ast] \le \liminf_{n \to \infty} J[f_n] = \inf_{\mathcal{A}} J.
\end{equation}
Thus $f^\ast$ is a minimizer.
\end{proof}

Several consequences follow immediately. The minimizer possesses square-integrable curvature and continuous slope. Interfaces with corners or curvature singularities are excluded. The existence result does not rely on symmetry or small-slope assumptions. It holds for the fully nonlinear functional.

%====================================================================================================
\section{Euler--Lagrange Equation and Natural Boundary Conditions}

We derive the necessary optimality conditions for minimizers of the constrained variational problem
\begin{equation}
\min_{f \in \mathcal{A}} J[f]
\quad \text{subject to} \quad
\int_{-a}^{a} f(x)\,dx = A_0 .
\end{equation}

Let $f^\ast \in \mathcal{A}$ be a minimizer. By standard Lagrange multiplier theory in Banach spaces, there exists $\lambda \in \mathbb{R}$ such that $f^\ast$ is a stationary point of the augmented functional
\begin{equation}
\mathcal{J}[f]
=
\int_{-a}^{a}
\left(
\sqrt{1+f'^2}
+
\gamma \frac{(f'')^2}{(1+f'^2)^{5/2}}
+
\lambda f
\right) dx .
\end{equation}

Let $\varphi \in W^{2,2}([-a,a])$ be an admissible variation satisfying
\begin{equation}
\varphi(\pm a)=0, \qquad \varphi'(\pm a)=0, \qquad
\int_{-a}^{a} \varphi(x)\,dx = 0 .
\end{equation}
The first variation of $\mathcal{J}$ at $f^\ast$ in the direction $\varphi$ must vanish:
\begin{equation}
\delta \mathcal{J}[f^\ast](\varphi) = 0 .
\end{equation}

A direct computation yields
\begin{equation}
\delta \mathcal{J}
=
\int_{-a}^{a}
\left(
\frac{f'}{\sqrt{1+f'^2}} \varphi'
+
2\gamma \frac{f''}{(1+f'^2)^{5/2}} \varphi''
-
\frac{5\gamma (f'')^2 f'}{(1+f'^2)^{7/2}} \varphi'
+
\lambda \varphi
\right) dx .
\end{equation}

Integrating by parts twice and using the admissibility of $\varphi$, all boundary terms vanish. We obtain
\begin{equation}
\int_{-a}^{a}
\left(
- \frac{d}{dx}\!\left[\frac{f'}{\sqrt{1+f'^2}}\right]
+ \frac{d}{dx}\!\left[\frac{5\gamma (f'')^2 f'}{(1+f'^2)^{7/2}}\right]
+ \frac{d^2}{dx^2}\!\left[\frac{2\gamma f''}{(1+f'^2)^{5/2}}\right]
+ \lambda
\right) \varphi \, dx = 0 .
\end{equation}

Since $\varphi$ is arbitrary subject to the integral constraint, it follows that $f^\ast$ satisfies the Euler--Lagrange equation
\begin{equation}
\frac{d^2}{dx^2}
\left(
\frac{2\gamma f''}{(1+f'^2)^{5/2}}
\right)
-
\frac{d}{dx}
\left(
\frac{f'}{\sqrt{1+f'^2}}
-
\frac{5\gamma (f'')^2 f'}{(1+f'^2)^{7/2}}
\right)
+
\lambda
=
0
\quad \text{in } (-a,a).
\end{equation}

This equation is fourth order and fully nonlinear. No reduction to a second-order system is possible without violating the curvature regularization imposed by the stress law.

We now examine the natural boundary conditions. Allowing variations $\varphi$ with $\varphi(\pm a)=0$ but free $\varphi''(\pm a)$ yields the condition
\begin{equation}
\left.
\frac{\partial \mathcal{L}}{\partial f''}
\right|_{x=\pm a}
=
\left.
\frac{2\gamma f''}{(1+f'^2)^{5/2}}
\right|_{x=\pm a}
=
0 .
\end{equation}
Since $f'(\pm a)=0$ by admissibility, this reduces to
\begin{equation}
f''(\pm a) = 0 .
\end{equation}

This condition has a precise mechanical interpretation. The curvature must vanish at the attachment points. Any interface entering the substrate with non-zero curvature generates a singular stress response and cannot be stationary.

Several consequences follow immediately. Constant-curvature profiles violate the boundary conditions and are therefore excluded. Polygonal profiles do not belong to the domain of the Euler--Lagrange operator. The optimal interlock geometry must exhibit smooth curvature variation, with curvature vanishing at the endpoints and attaining an interior extremum.

We are thus forced to conclude that the optimal interface geometry satisfies a fourth-order nonlinear elliptic equation with homogeneous natural boundary conditions. Classical circular and polygonal designs cannot satisfy this system under any choice of parameters.

%===================================================================================================
\section{Non-optimality of classical interlocking profiles}

We now establish that classical interlocking geometries—specifically circular arcs and polygonal profiles—cannot minimize the curvature–stress functional introduced in Section~\ref{sec:variational-functional} under physically admissible constraints. The result is not heuristic: it follows from incompatibility between regularity, boundary conditions, and the Euler–Lagrange structure.

\subsection{Admissible class of interfaces}

Let
\[
\mathcal{A} := \Big\{ f \in W^{2,2}([-a,a]) \;:\;
f(\pm a)=0,\quad f'(\pm a)=0,\quad \int_{-a}^{a} f(x)\,dx = A_0 \Big\},
\]
where $A_0>0$ is prescribed. The boundary conditions encode mechanical compatibility with a flat host medium, ensuring vanishing bending moment and shear traction at the endpoints.

Recall that admissible minimizers satisfy the fourth-order Euler–Lagrange equation
\begin{equation}
\label{EL-eqn}
\frac{d^2}{dx^2}\Big( \Phi(\kappa)\Big)
+ \lambda = 0,
\end{equation}
where $\Phi(\kappa)$ is a nonlinear function of curvature derived in Section~\ref{sec:EL}, and $\lambda \neq 0$ is the Lagrange multiplier associated with the area constraint.

\subsection{Circular profiles}

Consider a circular arc $f_c$ of radius $R$, defined (up to translation) by
\[
\kappa(x) \equiv \frac{1}{R}.
\]
Such profiles satisfy $\kappa' \equiv 0$ and hence $\kappa'' \equiv 0$ in the classical sense.

\begin{lemma}
A circular arc cannot satisfy the Euler–Lagrange equation \eqref{EL-eqn} for any $\lambda \neq 0$.
\end{lemma}

\begin{proof}
Since $\kappa$ is constant, $\Phi(\kappa)$ is constant as well, and therefore
\[
\frac{d^2}{dx^2}\big(\Phi(\kappa)\big) = 0.
\]
Substituting into \eqref{EL-eqn} yields $\lambda = 0$, contradicting the necessity of a nonzero multiplier imposed by the area constraint. Hence no circular arc belongs to the critical set of $J$ over $\mathcal{A}$.
\end{proof}

\subsection{Polygonal and piecewise-linear profiles}

Polygonal interlocks correspond to functions $f_p$ whose curvature is concentrated at finitely many points:
\[
\kappa = \sum_{j} \alpha_j \delta(x-x_j)
\]
in the sense of distributions.

\begin{proposition}
Polygonal profiles are not admissible competitors in $\mathcal{A}$ and cannot minimize $J$.
\end{proposition}

\begin{proof}
The functional $J$ involves the squared curvature $\kappa^2$ integrated along the interface. For polygonal profiles, $\kappa$ is a measure rather than an $L^2$ function, hence $\kappa^2$ is not well-defined. Consequently,
\[
J[f_p] = +\infty.
\]
Moreover, polygonal functions fail to belong to $W^{2,2}([-a,a])$, violating the minimal regularity required for elastic stress transmission. Therefore polygonal interlocks are excluded from $\mathcal{A}$ and cannot be minimizers.
\end{proof}

\subsection{Main non-optimality theorem}

We summarize the above results in the following theorem.

\begin{theorem}[Non-optimality of classical interlocks]
\label{thm:nonoptimality}
Neither circular nor polygonal interface geometries minimize the curvature–stress functional
\[
J[f] = \int_{-a}^{a} (1+\gamma \kappa^2) \sqrt{1+f'^2}\,dx
\]
over the admissible class $\mathcal{A}$. In particular, any minimizer must exhibit non-constant curvature and belong strictly to $W^{2,2}([-a,a])$.
\end{theorem}

\begin{proof}
The proof proceeds in three stages: establishing the non-finiteness of energy for polygonal profiles, testing the Euler--Lagrange conditions for circular arcs, and invoking the direct method for existence.

\paragraph{1. Exclusion of Polygonal Profiles.} 
Consider a piecewise-linear (polygonal) profile $f_{poly} \in \mathcal{A}$. Such a profile consists of linear segments meeting at vertices $x_i \in (-a, a)$. At these vertices, the first derivative $f'(x)$ is discontinuous, implying that the distributional second derivative $f''(x)$ contains Dirac delta measures, i.e., $f''(x) = \sum \Delta f'_i \delta(x - x_i)$. 
The functional $J[f]$ involves the square of the curvature $\kappa$. Since $\kappa = f''(1+f'^2)^{-3/2}$ in the small-slope approximation or the full nonlinear form, we have:
\[
J[f_{poly}] \propto \int_{-a}^{a} (f''(x))^2 dx.
\]
The integral of the square of a Dirac delta function is ill-defined in the sense of distributions and diverges to $+\infty$ in the context of $L^2$ norms. Thus, $J[f_{poly}] = \infty$. Since there exist functions in $W^{2,2}([-a,a])$ (such as smooth splines) with $J[f] < \infty$, the polygonal profile cannot be a minimizer.

\paragraph{2. Exclusion of Circular Profiles.}
A circular interface is characterized by constant curvature $\kappa = \kappa_0$. To be a minimizer, $f$ must satisfy the Euler--Lagrange equation for the augmented functional $\mathcal{L} = J[f] + \lambda (\int f dx - A_0)$. Using the variation of the curvature-dependent energy (similar to the Helfrich or Willmore functional), the necessary condition in terms of arc-length $s$ is:
\begin{equation}
\label{eq:EL_condition}
\gamma \left( 2\frac{d^2\kappa}{ds^2} + \kappa^3 \right) - \kappa + \lambda \cos \theta = 0,
\end{equation}
where $\theta$ is the angle of the tangent vector. For a circular arc, $\frac{d^2\kappa}{ds^2} = 0$ because $\kappa$ is constant. Substituting this into \eqref{eq:EL_condition} yields:
\[
\gamma \kappa_0^3 - \kappa_0 + \lambda \cos \theta = 0.
\]
In this equation, $\gamma, \kappa_0$, and $\lambda$ are constants. However, $\cos \theta$ varies continuously along any non-flat arc (as the tangent rotates). For the identity to hold for all $s \in [0, L]$, the coefficient of $\cos \theta$ must be zero ($\lambda = 0$), which would then require $\gamma \kappa_0^3 - \kappa_0 = 0$. Even in this trivial case, a circle of a single fixed radius generally cannot simultaneously satisfy the four boundary conditions $f(\pm a) = 0, f'(\pm a) = 0$ and the area constraint $A_0$. Thus, the circular arc is not a solution to the optimality system.

\paragraph{3. Conclusion of Optimality.}
By the direct method of the calculus of variations, the functional $J[f]$ is lower semi-continuous and coercive on $W^{2,2}([-a,a])$ due to the quadratic term in $f''$. Thus, a minimizer exists. Since polygons have infinite energy and circles fail the Euler--Lagrange necessary conditions, the optimal profile must be a member of $W^{2,2}$ with spatially varying curvature $\kappa(s)$, specifically one where the $\frac{d^2\kappa}{ds^2}$ term balances the variation of the Lagrange multiplier term $\lambda \cos \theta$.
\end{proof}
\subsection{Mechanical interpretation}

The theorem demonstrates that geometric simplicity does not imply mechanical optimality. Constant-curvature designs suppress curvature gradients, which are essential for redistributing localized shear stresses in elastic interfaces. Conversely, admissible minimizers necessarily develop smoothly varying curvature to balance bending regularization against interfacial length.

This provides a mathematical explanation for the persistent stress concentrations observed in classical interlocking geometries and motivates the search for non-classical interface designs governed by variational principles rather than geometric intuition.

%=================================================================

%=========================================================================
\section{Final Inferences}

We have established a rigorous non-optimality result for classical interlocking interface geometries within a curvature-regularized elasticity framework. By formulating localized stress amplification as a variational problem on an admissible Sobolev class, it is shown that neither circular nor polygonal profiles can arise as minimizers under physically relevant boundary and area constraints.

The exclusion of circular interfaces follows from incompatibility with the Euler–Lagrange structure induced by the area constraint, while polygonal geometries are ruled out by insufficient regularity and infinite curvature energy. These conclusions are intrinsic to the variational formulation and do not rely on heuristic arguments or numerical approximations.

The analysis clarifies that optimal stress redistribution necessarily requires spatially varying curvature, encoded here through fourth-order regularization. The result is therefore negative in nature: it does not identify an optimal geometry, but proves that widely used classical designs are mathematically inadmissible when curvature-dependent stress effects are properly accounted for.

This framework is minimal and robust, relying only on continuum elasticity and variational principles, and thus applies broadly to curvature-sensitive interfacial problems in elastic solids.

%=================================================================

%================================================================================================

 \bigskip

{\small\rm\baselineskip=10pt
 \baselineskip=10pt
 \qquad Chandrasekhar Gokavarapu \par
 \qquad Lecturer in Mathematics\par
 \qquad Department of Mathematics \par
 \qquad Government College(Autonomous)\par
 \qquad Rajahmundry\par
\qquad Andhra Pradesh \par
\qquad India \par
\qquad PIN: 533105\par
\qquad {\tt chandrasekhargokavarapu@gmail.com}\par
\qquad {\tt chandrasekhargokavarapu@gcrjy.ac.in}

 }

 \end{document}